\documentclass[12pt]{amsart}
\usepackage{amsmath,amssymb,amsbsy,amsfonts,amsthm,latexsym,
                        amsopn,amstext,amsxtra,euscript,amscd,bm}
\usepackage{url}
\usepackage[colorlinks,linkcolor=blue,anchorcolor=blue,citecolor=blue]{hyperref}
\usepackage{color}
\usepackage[english]{babel}

\begin{document}

\newtheorem{theorem}{Theorem}
\newtheorem{lemma}[theorem]{Lemma}
\newtheorem{algol}{Algorithm}
\newtheorem{cor}[theorem]{Corollary}
\newtheorem{prop}[theorem]{Proposition}

\newtheorem{proposition}[theorem]{Proposition}
\newtheorem{corollary}[theorem]{Corollary}
\newtheorem{conjecture}[theorem]{Conjecture}
\newtheorem{definition}[theorem]{Definition}
\newtheorem{remark}[theorem]{Remark}

 \numberwithin{equation}{section}
  \numberwithin{theorem}{section}

\newcommand{\comm}[1]{\marginpar{%
\vskip-\baselineskip 
\raggedright\footnotesize
\itshape\hrule\smallskip#1\par\smallskip\hrule}}

\def\sssum{\mathop{\sum\!\sum\!\sum}}
\def\ssum{\mathop{\sum\ldots \sum}}
\def\iint{\mathop{\int\ldots \int}}
\newcommand{\twolinesum}[2]{\sum_{\substack{{\scriptstyle #1}\\
{\scriptstyle #2}}}}

\def\cA{{\mathcal A}}
\def\cB{{\mathcal B}}
\def\cC{{\mathcal C}}
\def\cD{{\mathcal D}}
\def\cE{{\mathcal E}}
\def\cF{{\mathcal F}}
\def\cG{{\mathcal G}}
\def\cH{{\mathcal H}}
\def\cI{{\mathcal I}}
\def\cJ{{\mathcal J}}
\def\cK{{\mathcal K}}
\def\cL{{\mathcal L}}
\def\cM{{\mathcal M}}
\def\cN{{\mathcal N}}
\def\cO{{\mathcal O}}
\def\cP{{\mathcal P}}
\def\cQ{{\mathcal Q}}
\def\cR{{\mathcal R}}
\def\cS{{\mathcal S}}
\def\cT{{\mathcal T}}
\def\cU{{\mathcal U}}
\def\cV{{\mathcal V}}
\def\cW{{\mathcal W}}
\def\cX{{\mathcal X}}
\def\cY{{\mathcal Y}}
\def\cZ{{\mathcal Z}}

\def\C{\mathbb{C}}
\def\F{\mathbb{F}}
\def\K{\mathbb{K}}
\def\Z{\mathbb{Z}}
\def\R{\mathbb{R}}
\def\Q{\mathbb{Q}}
\def\N{\mathbb{N}}
\def\M{\textsf{M}}

\def\({\left(}
\def\){\right)}
\def\[{\left[}
\def\]{\right]}
\def\<{\langle}
\def\>{\rangle}

\def\vec#1{\mathbf{#1}}

\def\e{e}

\def\eq{\e_q}
\def\fS{{\mathfrak S}}

\def\bfalpha{{\boldsymbol \alpha}}

\def\lcm{{\mathrm{lcm}}\,}

\def\fl#1{\left\lfloor#1\right\rfloor}
\def\rf#1{\left\lceil#1\right\rceil}
\def\mand{\qquad\mbox{and}\qquad}

\def\jt{\tilde\jmath}
\def\ellmax{\ell_{\rm max}}
\def\llog{\log\log}

\def\Qbar{\overline{\Q}}

\title[Higher moments of theta functions]
{\bf Upper and lower bounds for higher moments of theta functions}
\author{Marc Munsch}
\address{CRM, Universit\'e de Montr\'eal, 5357 Montr\'eal, Qu\'ebec }
\email{munsch@dms.umontreal.ca}
\author{Igor E.~Shparlinski} 
\address{Department of Pure Mathematics, University of 
New South Wales, Sydney, NSW 2052, Australia}
\email{igor.shparlinski@unsw.edu}

\date{\today}


 \begin{abstract} We obtain optimal lower bounds for moments of theta functions. On the other hand, we also get new upper bounds on individual theta values and moments of theta functions on average
 over primes. The upper bounds are based on bounds of character sums and in particular on a modification of some recent results of M.~Z.~Garaev.
\end{abstract}

\bibliographystyle{alpha}
\maketitle

\section{Introduction}

\subsection{Background}

For a prime $p$, we denote by  $\cX_p$ the group of multiplicative characters 
modulo $p$ (we refer to~\cite{IwKow} for a background on characters). 
Denote by  $\cX_p^+$ the subgroup of $\cX_p$ of order $(p-1)/2$ consisting of 
even characters $\chi$ (those satisfying $\chi(-1) = 1$) and $\cX_p^-$ the subset of $\cX_p$ consisting of odd characters  $\chi$ (those satisfying $\chi(-1) = -1$). 
Furthermore, we use $\cX_p^*$ to denote the set of nonprincipal characters modulo $p$. 

For real $x> 0$ and $\eta\in\{ 0,1\}$ we set 
$$\varTheta_{p} (\eta, x,\chi)
=\sum_{n=1}^\infty \chi (n)n^{\eta}e^{-\pi n^2x/p}, 
\qquad \chi\in \cX_p. 
$$ 
We note that, if we set 
$\eta_{\chi}=1$ if $\chi$ is odd and $\eta_{\chi}=0$ otherwise,
then 
$$\varTheta_{p} (\eta_{\chi}, x,\chi)=\vartheta_{p} (x,\chi)
$$
is the classical theta-function of the character $\chi$, see~\cite{Dav} for a background and basic properties. 
\vspace{2mm}

When computing the root number of $\chi$ appearing in the functional equation of the associated Dirichlet $L$- function, the question of whether $\vartheta_{p} (1,\chi) \neq 0$ appears naturally (see~\cite{Lou99} for details). Numerical computations lead to the conjecture that it never happens if $\chi$ is primitive (see~\cite{CZ} for a counterexample with $\chi$ unprimitive). An algebraic approach based on class field theory, introduced in~\cite{paloma}, allows to prove partial results when $p=2l+1$ with $p$ and $l$ primes. Nevertheless, it does not give any results in the general case. Thus, a standard way to handle this type of problems analytically is to average over families of characters. As a consequence, we can deduce that the conjecture should be true for a good proportion of characters. The study of moments of theta functions has been initiated in~\cite{LoMu}, \cite{LoMu1}, \cite{Lou99} and several conjectures have been 
stated in~\cite{MuTh}. The aim of that paper is to pursue the analytical investigation about theta values.
 
 The paper is divided into two main parts. Firstly, we obtain optimal lower bounds for moments of theta functions. Secondly, we consider the two related problems of giving upper bounds for theta-values individually and on average.


 A standard problem in analytic number theory is to study moments of $L$-functions at their central point $s=1/2$. It is conjectured (see~\cite[Chapter~5]{MM}) that the moments satisfy the following asymptotic formula:
\begin{equation}
\label{eq:Conj L}
M_{2k}(p) =\sum_{\chi\in \cX_p}
\vert L(1/2,\chi )\vert^{2k}
\sim C_kp\log^{k^2}p,\hspace{1cm} C_k>0.
\end{equation} 
The conjecture~\eqref{eq:Conj L} holds  for $k=1$ 
(see~\cite[Remark~3]{Rama} and~\cite[Theorem~3]{Bal}, 
or~\cite{HB81a} for a more precise asymptotic expansion), 
and $k=2$
(see~\cite{HB81b}). 
Although there are numerical evidence and theorical reasons sustaining this conjecture, it remains open for $k\geq 3$.


However 
lower bounds of the expected order of magnitude
$$\sum_{\chi\in \cX_p}
\vert L(1/2,\chi )\vert^{2k}
\gg p\log^{k^2}p,$$
have been given by Rudnick and  Soundararajan~\cite{RS1}.

In a similar way, the moments of theta functions are defined in~\cite{LoMu}
as follows: 
$$S_{2k}^+(p) =\sum_{\chi\in \cX_p^+}\vert\vartheta (1,\chi)\vert^{2k} \mand 
S_{2k}^-(p) =\sum_{\chi\in \cX_p^-}\vert\vartheta (1,\chi)\vert^{2k}.$$ 
It is shown in~\cite{LoMu} that 
\begin{equation}
\label{eq:moment}
\begin{split}
&S_2^+(p)\sim\frac{p^{3/2}}{4\sqrt{2}},\qquad \quad S_2^-(p)\sim \frac{p^{5/2}}{16\pi\sqrt{2}}, \\
&S_4^+(p)\sim\frac{3p^{2}\log p}{16 \pi},\qquad S_4^-(p)\sim \frac{3p^4\log p}{512\pi^3}.
  \end{split}
\end{equation}
Note that the proof of the asymptotic formulas for the fourth moments in~\eqref{eq:moment}
is using some ideas of~\cite{ACZ}.

For higher moments, the trivial character gives the main contribution 
and it is shown in~\cite{MuTh} that 
$$
S_{2k}^+(p)\sim c_k p^k,\qquad  k\geq 3,
$$
where $c_k>0$ is a constant. Therefore, it is interesting to pull out the trivial character from 
the summation and define 
$$
T_{2k}^+(p)= \sum_{\chi\in \cX_p^+\backslash \chi{_0}}
\left|\vartheta_p (1,\chi)\right| ^{2k}.
$$ 
We conjecture, based on numerical computation and some theoretical support, that 
\begin{equation}
\label{eq:Conj TS}
\begin{split}
&T_{2k}^+(p)\sim a_kp^{k/2+1}\(\log p\)^{(k-1)^2},\\
& S_{2k}^-(p)\sim b_kp^{3k/2+1}\(\log p\)^{(k-1)^2},
\end{split}
\end{equation} 
for some
 positive constants $a_k$ and $b_k$, depending only on $k$.

 Indeed, this can be related to recent results of~\cite{HarperMaks} (see also~\cite{Heap}), where the authors obtain the asymptotic behaviour of a Steinhaus random multiplicative function (basically a multiplicative random variable whose values at prime integers are uniformly distributed on the complex unit circle). This can be viewed as a random model for $\vartheta_{p} (x,\chi)$. In fact, the rapidly decaying factor $e^{-\pi n^2/p}$ is mostly equivalent to restrict the sum over integers $n \le n_0(p)$ for some $n_0(p) \approx \sqrt{p}$ and the averaging behavior of $\chi(n)$ with $n\ll p^{1/2}$ is essentially similar to that of a Steinhaus random multiplicative function. Hence, these results are a good support for conjecture~\eqref{eq:Conj TS}. We  obtain results that confirm this heuristic.


\subsection{Our results}
%

We begin with a lower bound  of the right order of magnitude, 
which may be compared to the results obtained for $L$-functions
by  Rudnick and  Soundararajan~\cite{RS1,RS2}.
 
\begin{theorem}
\label{thm:low theta} 
For any fixed integer $k\geq 1$, we have $$T_{2k}^+(p) \gg p^{1+k/2}\log^{(k-1)^2} p \mand 
S_{2k}^{-}(p)
\gg 
p^{1+3k/2}\(\log p\)^{(k-1)^2}. $$  \end{theorem}

The proof of Theorem~\ref{thm:low char} is given in Section~\ref{sec:low theta}. Under the assumption of the Generalized Riemann Hypothesis, Munsch~\cite{Mun} obtains  the following upper bounds
\begin{equation*}
\begin{split}
&T_{2k}^+(p) \le p^{1+k/2}\log^{(k-1)^2+o(1)} p\\
&S_{2k}^{-}(p) \ll p^{1+3k/2}\(\log p\)^{(k-1)^2+o(1)}.
\end{split}
\end{equation*}
This greatly strengthens our belief in the conjectural asymptotic~\eqref{eq:Conj TS}.

Even though unconditionally we are far from getting upper bounds  of the expected order, we can obtain non trivial upper bounds for almost all primes and also on average over primes if $3\leq k\leq 6$.

It has been shown in~\cite{MuTh}
that for any nonprincipal character $\chi$ modulo $p$, the bound 
\begin{equation}
\label{eq:indiv}
\left|\vartheta_p(1,\chi)\right|  \le p^{\eta_{\chi}/2+7/16+o(1)}
\end{equation}
holds as $p\to \infty$. The same approach also applies to the 
more general sums $\varTheta_{p} (\eta, x,\chi)$ for any $\eta\in \{0,1\}$.

 We begin by some improvements for bounds of individual values of theta functions. We use a result of Garaev~\cite{Gar} to improve the bound~\eqref{eq:indiv} for almost all primes $p$.

\begin{theorem}
\label{thm:indiv} Let $X\geq 1$ be a sufficient large real number. For any   $\eta \in \{0,1\}$
we have 
$$
\sum_{p \le X} \max_{\chi\in \cX_p^*} \left|\varTheta_p(\eta, 1,\chi)\right|^8 
 \le X^{4\eta+4+o(1)}.
$$
\end{theorem}

Thus, we immediately derive:

\begin{cor}
\label{cor:indiv} Let $X\geq 1$ be a sufficient large real number. For all but $o(X/\log X)$ 
primes $ p \leq X$,  for any $\chi\in \cX_p$ and   $\eta\in\{ 0,1\}$
we have 
$$
\left|\varTheta_p(\eta, 1,\chi)\right|  \le p^{\eta/2+3/8+o(1)}.
$$
\end{cor}

Combining Theorem~\ref{thm:indiv} with the bounds~\eqref{eq:moment}, 
we obtain:

\begin{theorem}
\label{thm:moment}
Let $X\geq 1$ be a sufficient large real number. For  any fixed integer $k$ with 
$6 \ge k \ge 3$, we have
$$
\sum_{p \le X}  T_{2k}^+(p) \le  X^{3k/4+3/2+o(1)} 
\mand
\sum_{p \le X} S_{2k}^-(p)
\le 
X^{7k/4+3/2+o(1)}. 
$$
\end{theorem}

Finally, for almost all primes $p$, we have nontrivial estimates for 
arbitrary even moments. 

\begin{theorem}
\label{thm:moment a.a.}
Let $X\geq 1$ be a sufficient large real number. For all but $o(X/\log X)$ 
primes $p \leq X$, and any fixed integer $k \ge 1$, we have
$$
T_{2k}^{+}(p) \le p^{3k/4+1/2+o(1)}
\mand 
S_{2k}^{-}(p) \le p^{7k/4+1/2+o(1)}.
$$
\end{theorem}

Theorems~\ref{thm:indiv}, \ref{thm:moment}
and~\ref{thm:moment a.a.} are proven in Sections~\ref{sec:indiv}, \ref{sec:moment}
and~\ref{sec:moment a.a.}, respectively. 

As part of our main tools, we use various bounds on the character sums
 \begin{equation}
\label{eq:Sums St}
S_p(\chi;t)=\sum_{n\leq t}\chi(n).
\end{equation} 
which we define for 
$\chi\in \cX_p$ and a  real $t$.
As an application of our approach, we also obtain 
a lower bound on moments of the general Dirichlet polynomials
 \begin{equation}
\label{eq:Dir Poly}
\Xi_p(\chi;t) = \sum_{n\leq t}\xi_n \chi(n) 
\end{equation} 
with some real coefficient $\xi_n$ that are bounded away from zero.

\begin{theorem}
\label{thm:low char} For  $1 \le t < p$ and arbitrary coefficients $\xi_n \gg 1$, $n=1,2, \ldots$, 
we have 
$$
\sum_{\chi\in \cX_p^+\backslash \chi{_0}}\vert \Xi_p(\chi;t)\vert^{2k}  \gg p t^{k/2}\log^{(k-1)^2} p.
$$
\end{theorem}

The proof of Theorem~\ref{thm:low char} is given in Section~\ref{sec:low char}, immediately after 
the proof of   Theorem~\ref{thm:low theta} as it uses very similar ideas.


\section{Lower Bounds}

\subsection{Background on  the Riemann zeta-function}

First we recall the well known  Euler formula
\begin{equation}
\label{eq:Euler}
\prod_{p} \(1-\frac{1}{p^{s}}\)^{-1} = \zeta (s)
\end{equation}
for the Riemann zeta-function $\zeta(s)$  where the product is takem 
over all primes, that holds for any complex $s$ with $\Re s> 1$,
see~\cite[Equation~(1.12)]{IwKow}. 

We also need the following fact about the analytic 
properties of $\zeta(s)$. 

\begin{lemma}\label{lem:zeta}
For any complex $s = \sigma + i t$, with  
$|\Im s| = |t| \ge 2$ and   $1/2 \leq \Re s =\sigma <  1$,
we have 
$$ \zeta(s) \ll \vert \tau \vert^{c(1-\sigma)^{3/2}}\log \vert \tau\vert
\mand 
\zeta(1+it) \ll \log^{2/3}\vert t\vert,
$$
for some absolute constant $c>0$.
\end{lemma}

\begin{proof} See~\cite[Theorem~8.27]{IwKow} and~\cite[Corollary~8.28]{IwKow}, 
respectively. 
\end{proof}

The following consequence is important in verifying the assumption 
of~\cite[Theorem~1]{Bre1}, which is our main tool.

\begin{corollary}\label{cor:zeta}The function $(s-1)\zeta(s)$ verifies~\cite[Equation~(1.6)]{Bre1} in the range $\Re(s) >1/2$.
 \end{corollary}
\begin{proof} 
For $\vert \Im(s)\vert \geq 2$, this is a direct consequence of Lemma~\ref{lem:zeta} and the fact that for $\Re(s)>1$, the function $(s-1)\zeta(s)$ is bounded. In the bounded domain $\vert \Im(s)\vert \leq 2$, the function $s\zeta(1-s)$ is holomorphic thus bounded. The conclusion follows easily.
\end{proof}

\subsection{Bounds for the restricted divisor function}

The strategy behind the proof of  Theorem~\ref{thm:low theta} is to "mollify" theta moments by a short character sum. For that purpose, we need to have good estimates for sums of restricted divisor function. 
To do this, we employ some results of 
de~la~Bret\`eche~\cite{Bre1} on sums of arithmetical functions of many variables. These type of sums appears naturally when we count integer points of bounded height on some varieties. 
This has been used for example in~\cite{Bre2,Bre3} to prove Manin's conjecture in some special cases.

%

\begin{lemma}\label{lem:count}
For any integer  $k\geq 2$ and any real positive $\gamma_i \le 1$, $i=1, \ldots, k$, there exists a constant 
$\Gamma_k>0$ such that
\begin{equation}
\label{divisorrestricted}
\ssum_{\substack{a_i,b_i \le T^{\gamma_i},~i=1, \ldots, k\\
 a_1\cdots a_k= b_1\cdots b_k}} 1 \sim \Gamma_kT^{\gamma}\log^{(k-1)^2}T.
\end{equation}
where $\gamma = \gamma_1 + \ldots +\gamma_k$. 
\end{lemma}

\begin{proof} We make use of~\cite[Theorems~1 and~2]{Bre1}
and complete  the proof in the following three steps.

\subsubsection*{Step~1}
First we prove that Assumption~P1 
of~\cite[Theorem~1]{Bre1} is satisfied with 
\begin{equation}
\label{eq:alpha}
\bfalpha =(1/2,\ldots ,1/2)\in \R^{2k}.
\end{equation}

For  $2k$ positive  integers $(m_1,\ldots, m_{2k})$ we set 
$f(m_1,\ldots, m_{2k}) =1$ if 
$$
m_1\cdots m_k =m_{k+1}\cdots m_{2k},
$$ 
and $f(m_1,\ldots, m_{2k}) =0$, otherwise. 
We see that $f$ is multiplicative, 
that is, 
$$
f(m_1n_1, \cdots ,m_{2k}n_{2k}) 
=f(m_1,  \cdots ,m_{2k} )f(n_1, \cdots  ,n_{2k})
$$ 
whenever $\gcd (m_1\cdots m_{2k},n_1\cdots n_{2k})=1$.

For a vector $\vec{s} = (s_1, \ldots, s_{2k}) \in \C^{2k}$
of $2k$ complex numbers, we define the multiple Dirichlet series
$$
F(\vec{s})
=\ssum_{m_1,\ldots ,m_{2k}\geq 1} 
\frac{f(m_1,\ldots ,m_{2k})}{m_1^{s_1}\cdots m_{2k}^{s_{2k}}}.
$$
Let $d_k(m)$ be the number of ways of writing a positive integer $m\geq 1$ as a product of $k$ positive integers.
Since
$$\vert m_1^{s_1}\cdots m_{2k}^{s_{2k}}\vert
\leq (m_1\cdots m_{2k})^{\sigma(\vec{s})},
$$ 
where 
\begin{equation}
\label{eq:sigma s}
\sigma(\vec{s})=\min\{\Re~s_j~:~1\leq j\leq 2k\},
\end{equation}
we have 
$$
\ssum_{m_1,\ldots ,m_{2k}\geq 1} 
\left| \frac{f(m_1,\ldots ,m_{2k})}{m_1^{s_1}\cdots m_{2k}^{s_{2k}}}\right|
\leq\sum_{m\geq 1} \frac{(d_k(m))^2}{m^{2\sigma(\vec{s})}}
=\prod_{p\geq 2}\(
\sum_{a\geq 0} \frac{\binom{a+k-1}{k-1}^2}{p^{2a\sigma(\vec{s})}}
\),
$$
which proves the absolute convergence of $F(\vec{s})$ in the range 
$\sigma(\vec{s}) >1/2$ and  verifies  Assumption~P1 
of~\cite[Theorem~1]{Bre1} for $\bfalpha$ given by~\eqref{eq:alpha}.

\subsubsection*{Step~2} 
Let us recall some notations used in~\cite{Bre1}. We denote by $\mathcal{L}_{2k}(\mathbb{C})$ the space of linear forms
$$
\ell(X_1, \ldots, X_{2k})  \in \C[X_1, \ldots, X_{2k}]
$$
Following~\cite{Bre1}, we denote by $\vec{e_j}$, $j=1\ldots, 2k$, the canonical 
 basis of $\C^{2k}$ and $\left\{e_j^*\right\}_{j=1}^{2k}$ 
 the dual basis in $\mathcal{L}_{2k}(\mathbb{C})$.
Thus, in our case the linear form $e_j^*$ are explicitly 
given by 
$$
e_j^*(X_1, \ldots, X_{2k})  = X_j, \qquad j=1\ldots, 2k.
$$

We now prove that Assumptions~P2 and~P3 
of~\cite[Theorem~1]{Bre1} are satisfied with the $n=k^2$ linear forms 
$$
\ell^{(a,b)} =e_a^*+e_{k+b}^* = X_a + X_{k+b}, \qquad 1\leq a,b\leq k
$$ 
(here there is no linear forms $h^{(r)}$, in other 
words $\cR$ is empty in our version of the Assumption~P2 of~\cite[Theorem~1]{Bre1}).
Since $f$ is multiplicative, in this range, we have 
(where $p\geq 2$ runs over the prime numbers):
\begin{equation}
\label{eq:F Fp}F(\vec{s})
=\prod_pF_p(\vec{s}),
\end{equation}
with 
$$F_p(\vec{s})
=\ssum_{r_1,\ldots ,r_{2k}\geq 0} 
\frac{f(p^{r_1},\ldots,p^{r_{2k}})}{p^{r_1s_1+\cdots +r_{2k}s_{2k}}}
= 
\ssum_{\substack{r_1,\ldots ,r_{2k}\geq 0\\ r_1+\cdots +r_k =r_{k+1}+\cdots +r_{2k}}}
\frac{1}{p^{r_1s_1+\cdots +r_{2k}s_{2k}}}.$$ 
Now, 
$$F_p(\vec{s})
=1
+\sum_{a=1}^k\sum_{b=k+1}^{2k}
\frac{1}{p^{s_a+s_b}}
+\ssum_{\substack{r_1,\ldots ,r_{2k}\geq 0\\ r_1+\cdots +r_k =r_{k+1}+\cdots +r_{2k}\geq 2}}
\frac{1}{p^{r_1s_1+\cdots +r_{2k}s_{2k}}}$$
and, with $\sigma(\vec{s}) >0$,   where $\sigma(\vec{s})$ is given 
by~\eqref{eq:sigma s}, 
the absolute value of the third term of the right hand side of the above equality is bounded by
\begin{equation*}
\begin{split}
\ssum_{\substack{r_1,\ldots ,r_{2k}\geq 0\\ r_1+\cdots +r_k =r_{k+1}+\cdots +r_{2k}\geq 2}} &
\frac{1}{p^{(r_1+\cdots +r_{2k})\sigma(\vec{s})}}\\
& =\sum_{r\geq 2} \binom{r+k-1}{k-1}^2\frac{1}{p^{2r\sigma(\vec{s})}}
\ll \sum_{r\geq 2} \frac{r^{2k}}{p^{2r\sigma(\vec{s})}}. 
\end{split}
\end{equation*}
Hence
$$
F_p(\vec{s}) =1 +\sum_{a=1}^k\sum_{b=k+1}^{2k}
\frac{1}{p^{s_a+s_b}}
+O_A\(\frac{1}{p^{4\sigma(\vec{s})}}\)
$$
(where the constants in these $O_A$ depend on $A>0$).
Furthermore, for a given $A>0$ and for $\sigma(\vec{s}) \geq A$ we have 
$$
\prod_{a=1}^k\prod_{b=k+1}^{2k} \(1-\frac{1}{p^{s_a+s_b}}\)
=1 -\sum_{a=1}^k\sum_{b=k+1}^{2k} \frac{1}{p^{s_a+s_b}}
+O_A\(\frac{1}{p^{4\sigma(\vec{s})}}\).
$$
Therefore, we see that 
\begin{equation}
\label{eq:Fp prod}
F_p(\vec{s})\prod_{a=1}^k\prod_{b=k+1}^{2k} \(1-\frac{1}{p^{s_a+s_b}}\)
=1+O_A\(\frac{1}{p^{4\sigma(\vec{s})}}\).
\end{equation}
Taking the product over all primes and 
using the   Euler formula~\eqref{eq:Euler}, we
obtain from~\eqref{eq:F Fp} and~\eqref{eq:Fp prod} that for $\sigma(\vec{s})  > 1$ 
we have
\begin{equation}
\label{eq:F prod}
F(\vec{s}) =  \psi(\vec{s}) \prod_{a=1}^k\prod_{b=k+1}^{2k}\zeta(s_a+s_b), 
\end{equation}
where $\psi(\vec{s})$ is a holomorphic function 
for $\sigma(\vec{s}) \ge A$ for any fixed $A > 1/4$.

We now writing~\eqref{eq:F prod} as 
\begin{equation}
\label{eq:F sasb prod}
F(\vec{s}) \prod_{a=1}^k\prod_{b=k+1}^{2k}(s_a+s_b-1) 
=  \psi(\vec{s}) \prod_{a=1}^k\prod_{b=k+1}^{2k}(s_a+s_b-1)\zeta(s_a+s_b).
\end{equation}

Recalling Corollary~\ref{cor:zeta} 
we conclude that the left hand side of~\eqref{eq:F sasb prod} 
verifies~\cite[Equation~(1.6)]{Bre1} in the range 
$\sigma(\vec{s})\geq A$, for any $A>1/4$. Translating each coordinate by $1/2$, we see that 
$$
H(\vec{s}) 
=F(\vec{s} + \bfalpha)\prod_{a=1}^k\prod_{b=k+1}^{2k}(s_a+s_b)
$$
verifies~\cite[Equation~(1.6)]{Bre1} in the range 
$\sigma(\vec{s}) \geq B$ for any $B=A-1/2>-1/4$. 
Hence, Assumptions P2 and~P3  of~\cite[Theorem~1]{Bre1} are satisfied
for $H(\vec{s})$.

%

\subsubsection*{Step~3} 
It is easy to verify the last Assumption~P4 of~\cite[Theorem~1]{Bre1}. To conclude, we need a stronger version of~\cite[Theorem~1]{Bre1} which gives the exact power of $\log x$ in the asymptotic~\eqref{divisorrestricted}. Under the hypothesis of~\cite[Theorem~1]{Bre1}, we show that the extra condition~(iv) of~\cite[Theorem~2]{Bre1} is satisfied with 
\begin{equation}
\label{eq:beta}
\beta =(1,\ldots ,1)\in \R^{2k}.
\end{equation}

We start with the inequality
$$
H(0,\ldots ,0) =\prod_{p\geq 2}
\(1-\frac{1}{p}\)^{k^2} \(1+ \frac{k^2}{p}
+\sum_{r\geq 2}  \binom{r+k-1}{ k-1}^2 \frac{1}{p^{r}}
\)>0.
$$ 
Now, we are able to conclude the proof. With the notations of Step~2 above, we clearly have that the linear form $e_1^*+\cdots+e_{2k}^*$ lies in the positive convex cone of the linear forms $l^{(a,b)}$. Furthermore we have the equality 
$m={\mathrm{rank}}(\{\ell^{(a,b)};\ 1\leq a\leq k,\ k+1\leq b\leq 2k\}) = 2k-1$ and the result follows.
\end{proof}

%
%
%

Clearly, the constant $\Gamma_k$ of Lemma~\ref{lem:count}
can be evaluated explicitely.

\subsection{Moments of weighted character sums}
\label{sec:moment char}

Assume we are give some sequences $\xi_n$ of positive real 
numbers with $\xi_n \gg 1$. We consider the Dirichlet polynomials~\eqref{eq:Dir Poly}. 
Furthermore, we fix some $\varepsilon>0$, set 
$$
x=p^{\varepsilon},
$$ 
and define 
$$
A_\varepsilon(\chi)= S_p(x;\chi).
$$

We consider the following two sums
\begin{equation}
\label{eq:S1S2}
\Sigma_1=\sum_{\chi\in \cX_p^+\backslash \chi{_0}}\vert\Xi(\chi;t)\vert^{2}\vert A_\varepsilon(\chi)\vert^{2k-2} 
\quad \text{and}\quad
\Sigma_2 =\sum_{\chi\in \cX_p^+}\vert A_\varepsilon(\chi)\vert^{2k} .
\end{equation}

\begin{lemma}\label{lem:1-2moments weight} 
Let $t = p^{\tau}$ with some fixed $\tau \in (0,1)$, 
For any integer $k \ge 2$ and a  sufficiently small $\varepsilon>0$, there exists a 
constant $C(\varepsilon,k)>0$ depending only on $\varepsilon$ and $k$, 
such that:
$$
\Sigma_1 \gg p^{1+\varepsilon(k-1)} t \log^{(k-1)^2}p
\mand \Sigma_2 \sim C(\varepsilon,k) p^{1+\varepsilon k}\log^{(k-1)^2}p.
$$
 \end{lemma}

\begin{proof} We start with 
proving  asymptotic formula for $\Sigma_2$ as the proof is shorter.  
Using the orthogonality relations for even characters, we obtain:
$$
\Sigma_2= \frac{(p-1)}{2} \ssum_{\substack{a_1,b_1\ldots,a_k,b_k \leq x\\ 
a_1\cdots a_k\equiv\pm b_1\cdots b_k \pmod p}} 1.
$$ 
Choosing $\varepsilon < 1/k$  to ensure that $p^{\varepsilon k} < p$, we see the only possible solutions to
$$
a_1\cdots a_k\equiv\pm b_1\cdots b_k \pmod p
$$ 
are those with  
$ a_1\cdots a_k= b_1\cdots b_k$. 
Hence
$$\Sigma_2= \frac{(p-1)}{2} \ssum_{\substack{a_1,b_1\ldots,a_k,b_k \leq x \\ a_1\cdots a_k= b_1\cdots b_k}} 1$$ and we conclude using Lemma~\ref{lem:count}
 with $T =  p^\varepsilon  = x$ and $\gamma_1=\ldots = \gamma_k = 1$.

We complete the sum $\Sigma_1$ including the trivial character and bound its contribution trivially by 
$$
\vert \Xi(\chi_0;t)\vert^2  \vert A_\varepsilon(\chi_0)\vert^{2k-2} 
=O\(t^2 p^{2\varepsilon(k-1)}\).
$$ 
Thus
$$
\Sigma_1=
\sum_{\chi\in \cX_p^+}\vert\Xi(\chi;t)\vert^{2}\vert A_\varepsilon(\chi)\vert^{2k-2}+ O\(t^2 p^{2\varepsilon(k-1)}\) .
$$

Using the orthogonality of multiplicative characters again, we derive
\begin{equation}
\label{eq:S1 Expand}
\Sigma_1=\frac{(p-1)}{2} \ssum_{\substack{a,b\leq  t \\ a_1,b_1,\ldots,a_{k-1}, 
b_{k-1}\leq x \\ aa_1\cdots a_{k-1}=\pm bb_1\cdots b_{k-1} (\bmod q)}} \xi_a \xi_b +O\(t^2 p^{2\varepsilon(k-1)}\).
\end{equation}

Using  Lemma~\ref{lem:count} with  $T=p$, $\gamma_1 = \tau$  
and $\gamma_i =  \varepsilon$ for $2\leq i \leq k$, and taking a sufficiently 
small $\varepsilon > 0$,  we 
conclude the proof.
\end{proof}

\subsection{Moments  of the theta function}

We define $A_\varepsilon(\chi)$ as in Section~\ref{sec:moment char} 
and  consider the following 
\begin{equation}
\label{eq:fS}
\fS=\sum_{\chi\in \cX_p^+\backslash \chi{_0}}\vert\vartheta(1,\chi)\vert^{2}\vert A_\varepsilon(\chi)\vert^{2k-2}
\end{equation}
(with  the above choice  $x=p^{\varepsilon}$).

We use the following  approximation of 
$\vartheta(1,\chi)$ by a truncated sum, which easily
follows from the estimating the tail via the corresponding 
integral. 

\begin{lemma}
\label{lem:approx}
Let $\delta>0$ be a positive number.
 Then 
$$
\vartheta(1,\chi)=\sum_{n\leq p^{1/2+\delta}} 
\chi(n)e^{-\pi n^2/p} + O(p^{1/2}e^{-p^{\delta}}).
$$
\end{lemma}

The proof of Theorem~\ref{thm:low theta} is derived  from the
following  moment  estimate, which in turn follows from Lemmas~\ref{lem:1-2moments weight} 
and~\ref{lem:approx}

\begin{lemma}\label{lem:1-2moments} 
For any integer $k \ge 2$ and a  sufficiently small $\varepsilon>0$, there exists a 
constant $C(\varepsilon,k)>0$ depending only on $\varepsilon$ and $k$, 
such that:
$$
\fS\gg p^{3/2+\varepsilon(k-1)}\log^{(k-1)^2}p.
$$
 \end{lemma}

\begin{proof}  
Using the trivial bound
$$
\sum_{n\leq p^{2/3}} \chi(n)e^{-\pi n^2/p} \ll p^{2/3}
$$
by Lemma~\ref{lem:approx} (with $\delta = 1/6$) we have
$$
|\vartheta(1,\chi)|^2 =
\left \vert\sum_{n\leq p^{2/3}}
\chi(n)e^{-\pi n^2/p}  \right \vert^{2} + O\(p^{5/6}e^{-p^{1/6}}\).
$$
Therefore,
\begin{equation*}
\begin{split}
\fS=
 \sum_{\chi\in \cX_p^+\backslash \chi{_0}}&\left \vert \sum_{n\leq p^{2/3}} \chi(n)e^{-\pi n^2/p}\right \vert^{2}\vert A_\varepsilon(\chi)\vert^{2k-2}\\
& + O\(p^{5/6}e^{-p^{1/6}} 
\sum_{\chi\in \cX_p^+\backslash \chi{_0}} \vert A_\varepsilon(\chi)\vert^{2k-2}\) .
\end{split}
\end{equation*}
Then, by the  orthogonality of multiplicative characters (or using the
analogue of the asymptotic formula of Lemma~\ref{lem:1-2moments weight} 
 for $\Sigma_2$ with $k-1$ instead 
of $k$), we obtain 
\begin{equation*}
\begin{split}
\sum_{\chi\in \cX_p^+\backslash \chi{_0}} \vert A_\varepsilon(\chi)\vert^{2k-2} &
\le \sum_{\chi\in \cX_p^+} \vert A_\varepsilon(\chi)\vert^{2k-2} \\
& \ll p x^{k-1} \log^{(k-2)^2}p= p^{1+\varepsilon(k-1)}  \log^{(k-2)^2}p,
\end{split}
\end{equation*}
which implies the asymptotic formula
$$
\fS=
 \sum_{\chi\in \cX_p^+\backslash \chi{_0}}\left \vert \sum_{n\leq p^{2/3}} \chi(n)e^{-\pi n^2/p}\right \vert^{2}\vert A_\varepsilon(\chi)\vert^{2k-2} + O\(1\) .
$$
We complete the sum $\fS$ including the trivial character and bound its contribution trivially by 
$$
O\(|\vartheta(1,\chi_0)|^2 x^{2k-2} \)
= O\(px^{2k-2} \) =O\(p^{1+2\varepsilon(k-1)}\).
$$ 
Thus
$$
\fS=
 \sum_{\chi\in \cX_p^+}\left \vert \sum_{n\leq p^{2/3}} \chi(n)e^{-\pi n^2/p}\right \vert^{2}\vert A_\varepsilon(\chi)\vert^{2k-2} + O\(p^{1+2\varepsilon(k-1)}\) .
$$

Using the orthogonality of multiplicative characters again, we derive
$$
\fS=\frac{(p-1)}{2} \ssum_{\substack{a,b\leq  p^{2/3} \\ a_1,b_1,\ldots,a_{k-1}, 
b_{k-1}\leq x \\ aa_1\cdots a_{k-1}=\pm bb_1\cdots b_{k-1} (\bmod q)}} e^{-\pi (a^2+b^2)/p} + O\(p^{1+2\varepsilon(k-1)}\).
$$ 
 Hence, restricting the summation to $a,b \le p^{1/2}$ and using that 
 in this case $e^{-\pi (a^2+b^2)/p} \gg 1$ we obtain 
 $$
\fS\gg  p \ssum_{\substack{a,b \leq p^{1/2} \\ 
a_1,b_1,\ldots,a_{k-1},b_{k-1}\leq x  \\ 
aa_1\cdots a_{k-1}=bb_1\cdots b_{k-1}}}1 +O\(p^{1+2\varepsilon(k-1)}\).
$$
We now recall the formula~\eqref{eq:S1 Expand} and use  Lemma~\ref{lem:1-2moments weight} 
with $\tau=1/2$. The result now follows. 
\end{proof}

\subsection{Proof of Theorem~\ref{thm:low theta}}
\label{sec:low theta}

For the sums $\Sigma_2$ and $\fS$ given by~\eqref{eq:S1S2} and~\eqref{eq:fS}, 
with a sufficiently small $\varepsilon>0$, 
by the H\"older inequality, we get
$$ \fS^{k} \leq \Sigma_2^{k-1} \sum_{\chi\in \cX_p^+\backslash \chi{_0}}\vert \vartheta(1,\chi)\vert^{2k} ,
$$ 
We now recall Lemma~\ref{lem:1-2moments weight} that gives an upper bound $\Sigma_2$ and
Lemma~\ref{lem:1-2moments} that gives an lower  bound on $\fS$.  This yields to the lower bound
$$ \sum_{\chi\in \cX_p^+\backslash \chi{_0}}\vert \vartheta(1,\chi)\vert^{2k}  \gg p^{1+k/2}\log^{(k-1)^2} p.
$$ 
The proof in the case of odd characters follows exactly along the same lines.

\subsection{Proof of Theorem~\ref{thm:low char}} 
\label{sec:low char}

We fix $k$ with $(k-1)^2 > A$ and a sufficiently small $\varepsilon>0$.
We then  proceed as in the proof of 
 Theorem~\ref{thm:low theta} and obtain 
$$ \Sigma_1^{k} \leq \Sigma_2^{k-1} \sum_{\chi\in \cX_p^+\backslash \chi{_0}}\vert \Xi_p(\chi;t)\vert^{2k} ,
$$ 
where $\Sigma_1$ and $\Sigma_2$ are  given by~\eqref{eq:S1S2}. 
We now apply  Lemma~\ref{lem:1-2moments weight} 
and derive 
$$ 
\sum_{\chi\in \cX_p^+\backslash \chi{_0}}\vert \Xi_p(\chi;t)\vert^{2k}  \gg p t^{k/2}\log^{(k-1)^2} p, 
$$
which concludes the proof.

\section{Bounds of character sums}

\subsection{Preliminaries}

We extend the definitions of $\cX_p$ and $\cX_p^*$ to arbitrary 
 integers $k\ge 2$
and use $\cX_k$  and $\cX_k^*$ to denote the sets of all
all  characters and nonprincipal {\it primitive\/}
characters modulo $k$, respectively.
 
Similarly we defined $S_k(\chi;t)$ by~\eqref{eq:Sums St} 
for an arbitrary  integer $k\ge 2$ and
$\chi\in \cX_k$. 

We estimate the sums $S_k(\chi;t)$ given by~\eqref{eq:Sums St} for almost all moduli $k$ using 
the ideas of Garaev~\cite{Gar}. 

We now define the function $e(z) = \exp(2 \pi i z)$.
We recall, that for any integer $z$ and an odd integer  $Q = 2M +1\ge 1$,  we have 
the orthogonality relation
\begin{equation}
\label{eq:Orth}
\sum_{b=-M}^M \e(bz/Q) = \left\{\begin{array}{ll}
Q,&\quad\text{if $z\equiv 0 \pmod Q$,}\\
0,&\quad\text{if $z\not\equiv 0 \pmod Q$,}
\end{array}
\right.
\end{equation}
see~\cite[Section~3.1]{IwKow}. 

Furthermore, we also  need the bound
\begin{equation}
\label{eq:Incompl}
\sum_{n=U+1}^{U+V} \e(bn/Q) \ll  \min\left\{V, \frac{Q}{|b|}\right\},
\end{equation}
which holds for any integers  $b$, $U$ and $V\ge 1$ with $0 < |b| \le Q/2$,
see~\cite[Bound~(8.6)]{IwKow}.

First we recall the classical large sieve 
inequality, see~\cite[Theorem~7.11]{IwKow}:

\begin{lemma}
\label{lem:Large Sieve}
Let $a_1, \ldots, a_H$ be an arbitrary sequence
of complex numbers and let
$$
A = \sum_{h=1}^H |a_h|^2 \mand
T(u) = \sum_{h=1}^H a_h \exp(2 \pi i hu).
$$
Then, for an arbitrary $R\ge 1$, we have
$$
\sum_{1 \le r \le R} \sum_{\substack{v=1\\ \gcd(v,r) =1}}^{r}
\left|T(v/r)\right|^2 \ll
\( R^2+ H\) A.
$$
\end{lemma}

The link between multiplicative characters and exponential sums is given 
by the following well-known identity (see~\cite[Equation~(3.12)]{IwKow}) 
involving Gauss sums
$$
\tau_k(\chi) = \sum_{v=1}^k \chi(v) \e(v/k)
$$
defined for a character $\chi$ modulo an integer $k\ge 1$:

\begin{lemma}
\label{lem:tau chi}
For any primitive multiplicative character $\chi$ modulo $k$ and an integer $b$ with 
$\gcd(b,k) = 1$,  we have
$$
\chi(b) \tau_k( \overline\chi) = 
\sum_{\substack{v=1\\ \gcd(v,k) =1}}^{k}  \overline\chi(v) \e(bv/k), 
$$
where $\overline\chi$ is the complex conjugate character to $\chi$. 
\end{lemma}

By~\cite[Lemma~3.1]{IwKow} we also have:

\begin{lemma}
\label{lem:tau size}
For any primitive multiplicative character $\chi$ modulo an integer $k\ge 1$ we have
$$
|\tau_k(\chi)| = k^{1/2}.
$$
\end{lemma}

\subsection{Bounds for almost moduli} 

We  use some ideas of Garaev~\cite[Theorem~10]{Gar}, which we
adapt to our purposes and specific relations between the 
parameters.

\begin{lemma}
\label{lem:Garaev}
For $Q = X^{1/2 + o(1)}$, we have
$$
\sum_{k \in [X,2X]} \max_{\chi\in \cX_k^*} \max_{t \le Q}
\left|S_k(\chi;t)\right| ^{8} \ll X^{4+o(1)} .
$$
\end{lemma}

\begin{proof} We follow the ideas of Garaev~\cite[Theorem~3]{Gar}. 

For each $k\in [X, 2X]$ we choose a primitive multiplicative character $\chi_k$  modulo $k$
and $t_k \le Q$ 
such that with the largest values of 
$$
|S_k(\chi_k;t_k)|= \max_{\chi\in \cX_k^*} \max_{t \le Q}
\left|S_k(\chi;t)\right|
$$

Without loss of generality we can assume that $Q=2M+1$ is an 
odd integer. Then using~\eqref{eq:Orth}, for $t_k \le Q$ we write
\begin{eqnarray*}
S_k(\chi_k;t_k) &=&
 \sum_{m=1}^Q  \chi_k(m)\frac{1}{Q} \sum_{n=1}^{t_k} \sum_{b=-M}^{M} \e(b(m-n)/Q) \\
&=& \frac{1}{Q} \sum_{b=-M}^{M}   \sum_{n=1}^{t_k} \e(-bn/Q)
 \sum_{m=1}^Q  \chi_k(m) \e(bm/Q). 
\end{eqnarray*}
Recalling~\eqref{eq:Incompl}, we derive
$$
S_k(\chi_k;t_k)  \ll
 \sum_{b=-M}^{M} \frac{1}{|b|+1}   
\left|\sum_{m=1}^Q  \chi_k(m) \e(bm/Q)\right| .
$$
Therefore, writing 
$$
|b|+1 = \(|b|+1\)^{7/8} \(|b|+1\)^{1/8},
$$
 the H{\"o}lder inequality yields the bound
 \begin{equation}
\label{eq:Ub}
\sum_{k \in [X,2X]} \left|S_k(\chi_k;t_k) \right|^{8} \ll
(\log Q)^{7} \sum_{b=-M}^{M} \frac{1}{|b|+1}  U_b, 
\end{equation} 
where
$$
U_b =  \sum_{k \in [X,2X]} \left|\sum_{m=1}^Q  \chi_k(m) \e(bm/Q)\right|^{8}.
$$
We now note that 
$$
\(\sum_{m=1}^Q  \chi_k(m) \e(bm/Q)\)^4 = \sum_{h=1}^{H} \rho_{b}(h)  \chi_k(h) , 
$$
 where $H = Q^4$ and 
$$
\rho_{b}(h) = \sum_{\substack{m_1,m_2,m_3,m_4=1\\ m_1m_2m_3m_4 = h}}^Q 
\e(b( m_1+m_2+m_3 + m_4)/Q). 
$$
Using Lemma~\ref{lem:tau chi}, we  write
$$
\(\sum_{m=1}^Q  \chi_k(m) \e(bm/Q)\)^4 = \sum_{h=1}^{H} \rho_{b}(h)  
 \frac{1}{\tau_{k}( \overline\chi_k)}\sum_{\substack{v=1\\ \gcd(v,k) =1}}^{k} 
  \overline\chi_k(v) \e(hv/k).
$$
Changing the order of summation, 
by Lemma~\ref{lem:tau size} and the Cauchy inequality, we obtain,
$$ 
\left|\sum_{m=1}^M  \chi_k(m) \e(bm/Q)\right|^{8} \le
\sum_{\substack{v=1\\ \gcd(v,k) =1}}^{k} \left| \sum_{h=1}^{H} \rho_{b}(h) \e(hv/k)\right|^2.
$$
Therefore
$$
U_b \le   \sum_{k \in [X,2X]} \sum_{\substack{v=1\\ \gcd(v,k) =1}}^{k}\left| \sum_{h=1}^{H} \rho_{b}(h) \e(hv/k)\right|^2.
$$
Using the standard upper bound on the divisor function, see, 
for example,~\cite[Bound~(1.81)]{IwKow}, 
we conclude that 
$$
|\rho_{b}(h)| \le  \sum_{m_1m_2m_3m_4 = h} 1 =  h^{o(1)}
$$ 
as $h \to \infty$. 
Hence, we now derive from Lemma~\ref{lem:Large Sieve}
$$
U_b\le  \(X^2+ H \) H X^{o(1)} \le \(X^2+ Q^4 \) Q^4 X^{o(1)}\le X^{4+o(1)} , 
$$
which after substitution in~\eqref{eq:Ub} implies 
$$
\sum_{k \in [X,2X]}
 \left|S_k(\chi_k;t_k) \right|^{8} \ll X^{4+o(1)} 
$$
and concludes the proof.  
\end{proof} 


\section{Upper Bounds}
\subsection{Proof of Theorem~\ref{thm:indiv}}
\label{sec:indiv}

We first obtain a bound on 
for almost all primes in dyadic intervals.

\begin{lemma}
\label{lem:dyadic}
Let $X\geq 1$ be a sufficient large real number. For any   $\eta \in \{0,1\}$
we have 
$$
\sum_{p \in [X,2X]} \max_{\chi\in \cX_p^*} \left|\varTheta_p(\eta, 1,\chi)\right|^8  \le X^{4\eta+4+o(1)}.
$$
\end{lemma}
 
\begin{proof} 
Using partial summation, we obtain
$$\varTheta_p(\eta, 1,\chi)=\frac{2\pi}{p}\int_{1}^{+\infty} 
\(t^{1+\eta}+\eta\frac{p}{2\pi}\)S_p(\chi;t)e^{-\pi t^2/p} dt.
$$
where $S_p(\chi;t)$ is given by~\eqref{eq:Sums St}. 

First, we remark that it suffices to bound the above integral 
for $t\leq 2(X\log X)^{1/2}$. Indeed, we bound the tail for $t\ge 2(X\log X)^{1/2}$ trivially, 
using that we always have $|S_p(\chi;t)| \le p$, 
\begin{equation}
\label{eq:Tail}
\begin{split}
\frac{2\pi}{p} \int_{2(X\log X)^{1/2}}^{+\infty} & \(t^{1+\eta}+\eta\frac{p}{2\pi}\)
  S_p(\chi;t)e^{-\pi t^2/p} dt \\
& \ll  \int_{2(X\log X)^{1/2}}^{+\infty} t^{3}e^{-\pi t^2/p} dt  = I_1 + I_2  
\end{split}
\end{equation}
where 
$$
I_1 = \int_{2(X\log X)^{1/2}}^{X}t^{3}e^{-\pi t^2/p} dt \mand
I_2 = \int_{X}^{+\infty} t^{3}e^{-\pi t^2/p} dt.
$$ 
We now have the following elementary estimates:
\begin{equation*}
\begin{split}
I_1 &= X^3 \int_{2(X\log X)^{1/2}}^{X}e^{-\pi t^2/p} dt \ll 
X^3 \exp\(-\pi \frac{4 X \log X}{p}\)\\
& \qquad \qquad \qquad \qquad \qquad \qquad \qquad 
\qquad 
 \le X^3 \exp\(-\pi \log X\) \ll 1,\\
I_2 &\le  \int_{X}^{+\infty} e^{-0.5 \pi t^2/p} dt  \ll 1.
\end{split}
\end{equation*}

Substituting the above estimates of $I_1$ and $I_2$
in~\eqref{eq:Tail}, we now conclude 
\begin{equation*}
\begin{split}
\varTheta_p(\eta&, 1,\chi)\\  &\ll 
X^{-1} \int_1^{2(X\log X)^{1/2}}\(t^{1+\eta}+\eta\frac{p}{2\pi}\)
  |S_p(\chi;t)|e^{-\pi t^2/p} dt+ 1\\
  & \ll 
X^{-1}  \max_{t \le 2(X\log X)^{1/2}}\left|S_p(\chi;t)\right|
\int_1^{2(X\log X)^{1/2}}\(t^{1+\eta}+\eta\frac{p}{2\pi}\) dt+ 1.
\end{split}
\end{equation*}

Hence, 
\begin{equation}
\label{eq:prelim1}
\begin{split}
\varTheta_p(\eta, 1,\chi)  &\ll X^{-1}(X^{1/2}\log X)
((X \log X)^{(1+\eta)/2}+ \eta X) \\
& \qquad \qquad\qquad \qquad\qquad  \max_{t \le 2(X\log X)^{1/2}}\left|S_p(\chi;t)\right| + 1.
\end{split}
\end{equation}
We now remark that for $\eta \in \{0,1\}$ we have 
$$
\eta X\le   (X \log X)^{(1+\eta)/2}.
$$
Hence, we derive from~\eqref{eq:prelim1}  that 
\begin{equation}
\label{eq:prelim2}
\begin{split}
\varTheta_p(\eta, 1,\chi)  &\ll X^{-1} (X^{1/2}\log X)^{(3+\eta)/2} 
\max_{t \le 2(X\log X)^{1/2}}\left|S_p(\chi;t)\right| + 1\\
& = X^{\eta/2}  (\log X)^{(3+\eta)/2} \max_{t \le 2(X\log X)^{1/2}}\left|S_p(\chi;t)\right|
\end{split}
\end{equation}
(clearly the term $1$ can be dropped as, for example,  $S_p(\chi;1)=1$).
We now see from~\eqref{eq:prelim2} that 
\begin{equation*}
\begin{split}
\sum_{p \in [X,2X]} \max_{\chi\in \cX_p^*} & \left|\varTheta_p(\eta, 1,\chi)\right|^8  \\
& \le X^{4 \eta+o(1)}  \sum_{p \in [X,2X]} \max_{\chi\in \cX_p^*}  
\max_{t \le 2(X\log X)^{1/2}}\left|S_p(\chi;t)\right|^8, 
\end{split}
\end{equation*}
and using Lemma~\ref{lem:Garaev}  we conclude the proof.
 \end{proof}

We deduce Theorem~\ref{thm:indiv} from Lemma~\ref{lem:dyadic} by splitting $[1,X]$ in dyadic intervals $]X/2^{k+1},X/2^{k}]$.

\subsection{Proof of Theorem~\ref{thm:moment}}
\label{sec:moment}

For even characters we write, 
\begin{equation*}
\begin{split}
\sum_{p \le X}  \sum_{\chi\in \cX_p^+\backslash \chi{_0}}&\left|\vartheta_p (1,\chi)\right| ^{2k} \\
&\le \sum_{p \le X}   \max_{\chi\in \cX_p^+\backslash \chi{_0}}  \left|\vartheta_p(1,\chi)\right|^{2k-4} \sum_{\chi\in \cX_p^+\backslash \chi{_0}}\left|\vartheta_p (1,\chi)\right| ^{4}.
\end{split}
\end{equation*}
Using~\eqref{eq:moment}, we obtain,
$$
\sum_{p \le X}  \sum_{\chi\in \cX_p^+\backslash \chi{_0}}\left|\vartheta_p (1,\chi)\right| ^{2k} \le X^{2+o(1)} \sum_{p \le X}   \max_{\chi\in \cX_p^+\backslash \chi{_0}}  \left|\vartheta_p(1,\chi)\right|^{2k-4}.
$$
Finally, since $3 \le k \le 6$ then by the H{\"o}lder inequality
\begin{equation*}
\begin{split}
\sum_{p \le X}   \max_{\chi\in \cX_p^+\backslash \chi{_0}} & 
\left|\vartheta_p(1,\chi)\right|^{2k-4}\\
& \le \(\sum_{p \le X} 1\)^{(6-k)/4} 
\(\sum_{p \le X}   \max_{\chi\in \cX_p^+\backslash \chi{_0}} 
 \left|\vartheta_p(1,\chi)\right|^{8}\)^{(k-2)/4}, 
\end{split}
\end{equation*}
and using Theorem~\ref{thm:indiv}, after simple calculations
we obtain the result for the even characters. A similar argument 
also implies the desired estimate for the odd characters.

\subsection{Proof of Theorem~\ref{thm:moment a.a.}}
\label{sec:moment a.a.}

For $k=1$ and $k=2$ the result is contained in~\eqref{eq:moment}.
 Let $X\geq 1$ be a sufficient large real number. For all but $o(X/\log X)$ 
primes $ p \leq X$, we see from Corollary~\ref{cor:indiv} that for any $\chi\in \cX_p$ 
we have 
$$
 \max_{\chi\in \cX_p^+\backslash \chi{_0}} 
 \left|\vartheta_p(1,\chi)\right| \le p^{3/8+o(1)}.
$$
So for these primes, for $k\ge 3$, we have
$$
 \sum_{\chi\in \cX_p^+\backslash \chi{_0}}\left|\vartheta_p (1,\chi)\right| ^{2k} 
\le  p^{3k/4+o(1)}  \sum_{\chi\in \cX_p^+\backslash \chi{_0}}
\left|\vartheta_p (1,\chi)\right| ^{4} .
$$
Using~\eqref{eq:moment}, we obtain the result for the even characters. A similar argument 
also implies the desired estimate for the odd characters.

%
%
%
%
%
%
%
%
%
%
%

\section{Concluding remarks}

We remark that under the Generalised Riemann Hypothesis, 
instead of Lemma~\ref{lem:Garaev} we can use 
the well-known bound 
\begin{equation}
\label{eq:GRH Bound}
\max_{\chi\in \cX_p^*} \left|S_p(\chi;t)\right| \le t^{1/2}p^{o(1)} ,
\end{equation} 
see~\cite[Section~1]{MoVau}; 
it can also be  derived   from~\cite[Theorem~2]{GrSo}. 

Hence, substituting~\eqref{eq:GRH Bound} in~\eqref{eq:prelim2} (with $X=p$) we obtain a more realistic individual bound
$$
\varTheta_p(\eta, 1,\chi) \le   p^{1/4+\eta/2 +o(1)}. 
$$ It is worthwhile to notice that this is consistent with the asymptotic conjectural formula (\ref{eq:Conj TS}).

%

\section*{Acknowledgements}

The authors are grateful to Driss Essouabri, Youness Lamzouri, Stephane Louboutin and 
Maksym Radziwi\l\l\ for very helpful discussions.

 During the preparation of this work M.~Munsch was supported by a postdoctoral grant in CRM, Montreal, under the supervision of Andrew Granville and Dimitris Koukoulopoulos
and I.~E.~Shparlinski was supported in part by the Australian Research Council
Grant DP140100118.

\end{document}